\documentstyle[leqno,amssymb]{article}

 \def\qed{\unskip\quad \hbox{\vrule\vbox to 6pt {\hrule width
          4pt\vfill\hrule}\vrule} }
          \newcommand{\bez}{\nopagebreak\hspace*{\fill}
          \nolinebreak$\qed$\vspace{5mm}\par}
          \newenvironment{proof}{\vspace{1ex}
               \vspace*{10mm}\vspace{-10mm}\par\noindent{\bf

               Proof.}\nopagebreak
                    \par}{\nopagebreak\linebreak[0]\hspace*{\fill}
 $\qed$\vspace{5mm}\pagebreak[0]\vspace{3ex}\par}

                         \newtheorem{Th}{Theorem}
                         \newtheorem{Prop}{Proposition}
                         \newtheorem{Lemma}{Lemma}
                         \newtheorem{Remark}{Remark}
                         
                         \newtheorem{Def}{Definition}
                         \newtheorem{Cor}{Corollary}

\newcommand{\bs}{{\Bbb{S}}}

\newcommand{\cc}{{\cal C}}
\newcommand{\cs}{{\cal S}}

\newcommand{\ci}{{\cal I}}

\newcommand{\ce}{{\cal E}}

\newcommand{\cb}{{\cal B}}
\newcommand{\ca}{{\cal A}}

\newcommand{\cf}{{\cal F}}

\newcommand{\cm}{{\cal M}}
\newcommand{\cp}{{\cal P}}

\newcommand{\R}{{\Bbb{R}}}

\newcommand{\C}{{\Bbb{C}}}
\newcommand{\Z}{{\Bbb{Z}}}

\newcommand{\E}{{\Bbb{E}}}

\newcommand{\xbm}{(X,{\cal B},\mu)}

\newcommand{\zdr}{(Z,{\cal D},\rho)}
\newcommand{\ycn}{(Y,{\cal C},\nu)}

\newcommand{\tf}{T_{\varphi}}

\newcommand{\tfs}{T_{\varphi,{\cal S}}}

\newcommand{\ov}{\overline}
\newcommand{\beq}{\begin{equation}}
\newcommand{\eeq}{\end{equation}}
\newcommand{\vep}{\varepsilon}
\newcommand{\va}{\varphi}

\newcommand{\ot}{\otimes}
\newcommand{\la}{\lambda}

\newcommand{\Aut}{{\rm Aut}}

\begin{document}

\title{Relatively finite measure-preserving extensions
and lifting multipliers by Rokhlin cocycles}
\author{T. Austin\thanks{Research partially supported by
MSRI (Berkeley) program ``Ergodic Theory and
Additive Combinatorics"}  \and M. Lema\'nczyk\thanks{Research
partially supported by Polish MNiSzW grant N N201 384834; partially
supported by Marie Curie ``Transfer of Knowledge'' EU program --
project MTKD-CT-2005-030042 (TODEQ) and MSRI (Berkeley) program
``Ergodic Theory and Additive Combinatorics"}}

\maketitle

\thispagestyle{empty} {\em Dedicated to Stephen Smale in
recognition of his contributions to topology
 and dynamical systems}

\vspace{5ex}

\noindent

\vspace{5pt}

\noindent

\vspace{5ex}

 \normalsize

\begin{abstract}We show that under some natural ergodicity assumptions
extensions given by Rokhlin cocycles lift the multiplier property
if the associated locally compact group extension has only
countably many $L^\infty$-eigenvalues. We make use of some analogs
of basic results from the theory of finite-rank modules associated
to an extension of measure-preserving systems in the setting of a
non-singular base.
\end{abstract}

\section*{Introduction}

Our main concern in this paper is with measure-preserving
automorphisms of a standard Borel probability space $\xbm$.  We
will denote by $\Aut_0\xbm$ the Polish group of all such
automorphisms up to almost-everywhere agreement.

Assume that $T \in \Aut_0\xbm$ is ergodic. If $\ov{T}$ acting on
$(\ov{X},\ov{\cb},\ov{\mu})$ is an ergodic extension of $T$ then
the classical Abramov-Rokhlin Theorem states that we can assume
that $(\ov{X},\ov{\cb},\ov{\mu})=\xbm\ot\ycn$ and
$$\ov{T}(x,y)=(Tx,\theta_x(y)),$$
where $\theta:X\to\Aut_0\ycn$ is a so-called {\em Rokhlin
cocycle}.

Moreover, as noticed in \cite{Da-Le}, any $\Aut_0\ycn$-valued
cocycle $\theta$ is cohomologous to a special cocycle. Namely,
there exist a locally compact amenable group $G$ together with a
cocycle $\va:X\to G$ and a measurable $G$-action $\cs=(S_g)_{g\in
G}$ on $\ycn$ such that $\theta$ is  cohomologous to the cocycle
$x\mapsto S_{\va(x)}$. The resulting automorphism will be denoted
by $T_{\va,\cs}$: that is,
$$T_{\va,\cs}(x,y)=(Tx,S_{\va(x)}(y)).$$
It follows that the study of ergodic extensions of a given
transformation $T$ is reduced to that of systems of the form
$T_{\va,\cs}$.

Such extensions have recently been examined in the context of
lifting disjointness properties \cite{Da-Le}, \cite{Gl},
\cite{Gl-We}, \cite{Le-Le}, \cite{Le-Pa}.  In particular, the
investigations \cite{Da-Le} and \cite{Le-Pa} led to sufficient
conditions on $\va$ and $\cs$ so that if $T$ and $R$ are disjoint
in Furstenberg's sense \cite{Fu} (denoted $T\perp R$) then also
$T_{\va,\cs}\perp R$.

A more subtle problem concerns the lifting of the multiplier
property. Given a class $\cf\subset\Aut_0\xbm$, we denote by
$\cf^\perp$ the class of automorphisms disjoint from all systems in
$\cf$. An automorphism $T$ is said to be a {\em multiplier} of
$\cf^{\perp}$ (denoted $T\in \cm(\cf^\perp)$) if for each
automorphism $S\in \cf^\perp$ and any ergodic joining $T\vee S$ we
also have $T\vee S\in \cf^\perp$. A central result of \cite{Le-Pa}
provided an example of an ergodic automorphism $T_{\va,\cs} \in
\rm{WM}^\perp$ which was not a multiplier of WM$^\perp$, and which
could then be used to answer negatively the question, originally
posed by Glasner in~\cite{Gl}, of whether the class $\rm{WM}^\perp$
is closed under ergodic joinings. In that construction $T$ was an
ergodic rotation and the skew product $T_\va$ acting on $X\times G$
by the formula $T_\va(x,g)=(Tx,\va(x)g)$ (which preserves the {\bf
infinite} measure $\mu\ot\la_G$, where $\la_G$ denotes a Haar
measure of $G$) had uncountably many $L^\infty$-eigenvalues.  This
property was crucial for the arguments of~\cite{Le-Pa}, all earlier
constructions over a rotation for which $T_\va$ had only countably
many $L^\infty$-eigenvalues having led to multipliers of WM$^\perp$
(see~\cite{Gl-We}). The relationship between \cite{Gl}, \cite{Gl-We}
and \cite{Le-Pa} was explained in \cite{Da-Le}, where the main
result asserted that given an ergodic rotation $T$ and an ergodic
cocycle $\va$ such that $T_\va$ has countable $L^\infty$-spectrum,
the automorphism $T_{\va,\cs}$ is always a multiplier of WM$^\perp$.

In the present paper we use an alternative approach to the
multiplier property to prove a  generalization of the result from
\cite{Da-Le} (see Theorem~\ref{main} below):

{\em Assume that $T$ is ergodic and that $\va:X\to G$ is ergodic,
and that these are such that $T_{\va,\cs}$ is ergodic and
$T_{\va}$ has countably many $L^\infty$-eigenvalues. If in
addition $T\in\cm({{\cf}^\perp})$ for some class $\cf\subset {\rm WM}$,
then also $T_{\va,\cs}\in\cm({\cf}^\perp)$. In particular, the
result holds for $\cf={\rm WM}$. }

We will also show that this theorem gives a criterion for the
lifting of the {\em joining primeness property} introduced in
\cite{Le-Pa-Ro}, and from this that we can produce examples of
systems with the joining primeness property which are not distally
simple.

In the course of proving these results, we are naturally drawn
into an examination of certain relatively finite measure
preserving extensions of underlying non-singular base systems.  In
this setting we make use of some analogs of basic results from the
theory of finite-rank modules associated to an extension of
measure-preserving systems (see, for example, Chapter 9 of
Glasner~\cite{Gla03}), whose proofs turn out to be easily
transplantable to the setting of a non-singular base and which may
be of some independent interest.

\section{Relatively finite measure-preserving factors of singular
automorphisms}

The later sections of this paper will rely on some basic notions
concerning not only measure-preserving systems, but also certain
non-singular systems. We denote by $\Aut\xbm$ the group of all
non-singular automorphisms (considered up to almost-everywhere
agreement) of a standard Borel probability space $\xbm$.

Assume that $T\in\Aut\xbm$. Denote by $\cp(X)$ the set of
probability measures on $(X,\cb)$. Let
$$
\alpha(x)=\frac{d\mu\circ T}{d\mu}(x)$$ denote the Radon-Nikod\'ym
derivative of $\mu\circ T$ with respect to $\mu$ (the
non-singularity of $T$ gives that $\mu\circ T\sim\mu$). Let
$$\cp_\alpha=\{\lambda\in\cp(X):\:\frac{d\la\circ
T}{d\la}(x)=\alpha(x)\;\;\lambda-\mbox{a.e.}\}.$$ Then
$\mu\in\cp_\alpha$ and $\cp_\alpha$ is a simplex whose subset of
extremal points $\ce_\alpha \subset \cp_\alpha$ contains precisely
the ergodic members of $\cp_\alpha$.

Each measure $\nu\in \cp_\alpha$ gives rise to a non-singular
automorphism $(T,X,\cb,\nu)$.  As is well-known, the simplex
$\cp_\alpha$ now appears in the ergodic decomposition of
$(T,\nu)$.

\begin{Th}[\cite{Gr-Sch}]\label{decomp} If $\nu\in\cp_\alpha$ then the ergodic
decomposition of $(T,\nu)$ is given by \beq\label{decomp1}
\nu=\int_{\cp_\alpha}\epsilon\,dQ(\epsilon)\eeq for some
probability measure $Q$ on $\cp_\alpha$. \bez
\end{Th}

Returning to our original non-singular system $(T,X,\cb,\mu)$,
assume that $\ca\subset\cb$ is an invariant $\sigma$-algebra (a
{\em factor} of $T$). Following \cite{Da-Le} we say that the
factor $T|_{\ca}$ is {\em relatively finite measure-preserving}
({\em r.f.m.p.}) if $\alpha = \frac{d\mu\circ T}{d\mu}$ is
$\ca$-measurable. Let $(S,Y,\cc,\nu)$ be the factor system
$(T|_\ca,X/\ca,\ca,\mu|_\ca)$ given by $\ca$ and let $\pi:X\to Y$
be the factor map. Then the fact that $S$ is an r.f.m.p.\ factor
of $T$ is expressed by the equality (see \cite{Da-Le})
\beq\label{rfmp2} \mu_{S\pi(x)}=\mu_{\pi(x)}\circ T^{-1},\eeq
where $\mu=\int_{Y}\mu_y\,d\nu(y)$ is the distintegration of $\mu$
over $\nu$.

Notice that in the ergodic decomposition~(\ref{decomp1}) the
Radon-Nikod\'ym derivative of $Q$-a.e.\ ergodic component $\epsilon$
is also $\alpha$.  From this we can quickly deduce the following.

\begin{Cor}\label{rfmp1}
If $T|_\ca$ is ergodic then for $Q$-a.e.\ ergodic component
$\epsilon$ in~(\ref{decomp1}) the system $(T|_\ca,\mu|_\ca)$ is
also a factor of $(T,X,\cb,\epsilon)$ and moreover the factor $\ca
\subset \cb$ is also r.f.m.p. for $(T,X,\cb,\epsilon)$.
\end{Cor}

\begin{proof}
We only need to prove that the action of $T$ on $\ca$ with the
measure $\mu|_\ca$ is a factor of $(T,\epsilon)$. As remarked in
\cite{Da-Le}, this follows directly from the formula~(\ref{rfmp2})
and the uniqueness of ergodic decomposition.
\end{proof}

We now recall a non-singular version of Rokhlin's Theorem (see
\cite{Ra}). Let $T$ be an ergodic non-singular automorphism of a
standard probability Borel space $\xbm$ and let $S$ acting on
$\ycn$ be its factor (given by an invariant sub-$\sigma$-algebra
$\ca\subset\cb$). Then, up to measure space isomorphism,
$\xbm=\ycn\otimes\zdr$ for another standard Borel space $\zdr$ and
in these new ``coordinates'' we can express $T=S_\Theta$, where
$$
S_\Theta(y,z)=(Sy,\Theta(y)(z))$$ and $\Theta:Y\to \Aut\zdr$ is a
Rokhlin cocycle with values in the group of non-singular
automorphisms of $\zdr$. The automorphism $S_\Theta$ has still $S$
as its factor. If $S$ was an r.f.m.p.\ factor of $T$, it is now an
r.f.m.p.\ factor of $S_\Theta$. However the disintegration of
$\nu\ot\rho$ over $\nu$ is trivial and now the
equality~(\ref{rfmp2}) asserts that $\rho=\rho\circ
\Theta(y)^{-1}$ for $\nu$-a.e.\ $y\in Y$ (that is, the fiber
automorphisms are measure-preserving).

\begin{Cor}\label{rfmpwn} Under the above notation, if $T$ is ergodic and $S$ is an
r.f.m.p.\ factor  then (up to isomorphism) in the skew product
representation of $T$ over $S$ the Rokhlin cocycle $\Theta$ takes
values in the group $\Aut_0\zdr$ of measure-preserving automorphisms
of $\zdr$.\bez
\end{Cor}

\section{Finite-rank modules over a non-singular base}

Our later proofs will make use of a version of the theory of
invariant finite-rank modules over a factor, adapted to the setting
of an r.f.m.p. extension of a non-singular base system. This
machinery is well-known as applied to extensions of
measure-preserving systems, and much of it can simply be carried
over to our scenario unchanged.  For this purpose we refer the
reader to Chapter 9 of Glasner's book~\cite{Gla03} for a thorough
account, and will first recall some basic results from there.  Let
us stress that in this section we will make repeated appeal to the
fact that our transformation is ergodic.  It seems likely that an
analogous theory can be developed without this restriction (as has
been done for the measure-preserving case in~\cite{Aus--dirint}),
but we do not explore this rather technical subject here.  Later we
will apply the results of this section to the f.m.p. extensions
described in the Introduction in a somewhat complicated way; in
order to lighten our notation, for the duration of the present
section we will instead let $T$ be a non-singular transformation on
$\xbm$ that is r.f.m.p. over $\ca \subset \cb$, with the warning
that this does not correspond to the transformation $T$ discussed in
the Introduction.

Given a finite measure space $\xbm$ and a sub-$\sigma$-algebra
$\ca \subset \cb$, an {\em $\ca$-module} is a subspace $M \subset
L^2\xbm$ such that whenever $f\in M$ and $h \in L^\infty\xbm$ is
$\ca$-measurable, then also $h\cdot f\in M$.

Our picture of such modules becomes rather clearer when we invoke
the Rokhlin representation of $\ca \subset \cb$.  For the
remainder of this section let us write $\xbm = \ycn\otimes \zdr$,
so that our transformation $T$ becomes $S_\Theta$ as above.  We
can now clearly identify the $\ca$-measurable members of
$L^\infty\xbm$ with $L^\infty\ycn$, and we will henceforth
abusively write simply that $L^\infty\ycn \subset L^\infty\xbm$.
Writing $\pi:X\to Y$ for the canonical factor map, we will also
refer instead to a $\pi$-module.

A $\pi$-module $M$ can now be easily identified with a measurable
bundle of Hilbert subspaces $M_y \subset L^2\zdr$ indexed by $y \in
Y$, so that $M$ is the space of measurable sections of this bundle.
This module is of {\em finite-rank} if ${\rm dim}\,M_y < \infty$
almost everywhere, and is of {\em rank $r$} if ${\rm dim}\,M_y = r$
almost everywhere. For this special class of modules we have the
following relativized ability to select orthonormal bases.

\begin{Lemma}[Lemma 9.4 of~\cite{Gla03}]
If $M$ is a $\pi$-module of rank $r < \infty$ then there is an
$r$-tuple $\phi_1$, $\phi_2$, \ldots, $\phi_r$ of members of $M$
such that $\E(\phi_i\ov{\phi}_j\,|\,\ca) \equiv \delta_{ij}$ for
$1\leq i,j \leq r$ and such that any $f \in M$ can be expressed as
$\sum_{i=1}^r h_i\cdot \phi_i$ for some $h_1$, $h_2$, \ldots, $h_r
\in L^\infty\ycn$.

We refer to such a tuple of functions as a {\em fiberwise
orthonormal basis of $M$}. \bez
\end{Lemma}

Our interest will be in finite-rank $\pi$-modules that are invariant
under $S_\Theta$.  In order to work with these in the setting of a
non-singular base system $(S,Y,\cc,\nu)$ we first introduce a more
refined residence for them within $L^2\xbm$.

\begin{Def}\em
We write $L^\infty_\ca L^2\xbm$ for the subspace of those measurable
functions $f \in L^2\xbm$ such that the conditional expectation
satisfies $\E(f\,|\,\ca) \in L^\infty(\mu|_\ca)$.  We equip this
function space with the norm
\[\|f\|_{L^\infty_\ca L^2} := \|\E(|f|^2\,|\,\ca)\|_\infty.\]

Under the identification $\xbm = \ycn\otimes\zdr$, it is clear that
$L^\infty_\ca L^2\xbm$ is identified with
$L^\infty(Y,\cc,\nu;L^2\zdr)$.
\end{Def}

Now, we have represented the transformation $T$ as $S_\Theta$, and
so its action on $L^\infty(Y,\cc,\nu;L^2\zdr)$ can be written as
\[\xi\circ S_\Theta(y,z) = \xi(Sy,\Theta(y)(z)) =
U_{\Theta(y)}(\xi(Sy,\,\cdot\,))(z),\]
where $U_{\Theta(y)}\in U(L^2\zdr)$ is the unitary operator
associated to the transformation $\Theta(y) \in \Aut_0\zdr$ by the
Koopman representation.  In order to work with this
representation, given a function $f \in L^2\xbm$ we will sometimes
write $f|_y$ for the restriction $f(y,\,\cdot\,)$ to the fiber
$\pi^{-1}\{y\}$, regarded as a member of $L^2\zdr$.

The following is now immediate from this representation.

\begin{Lemma}
A $\pi$-submodule $M \subset L^2\xbm$ is $S_\Theta$-invariant if and
only if $M_{Sy} = U_{\Theta(y)}(M_y)$ for almost every $y$. \bez
\end{Lemma}

\begin{Prop}\label{prop:finite-rank-ess-bdd}
If $M \subset L^2\xbm$ is an $S_\Theta$-invariant finite-rank module
over $\ca \subset \cb$, then its rank is almost surely constant on
$Y$, and if $\phi_1$, $\phi_2$, \ldots, $\phi_r$ is a fiberwise
orthonormal basis for $M$ then $\sum_{j=1}^r|\phi_j(x)|^2$ is almost
surely equal to $r$. As a result any such $M$ is contained in
$L^\infty\xbm$.
\end{Prop}

\begin{proof}
The assertion that ${\rm dim}\,M_y$ is almost surely constant
follows simply because on the one hand $M_y$ and hence also ${\rm
dim}\,M_y$ vary measurably with $y$, but on the other ${\rm
dim}\,M_{Sy} = {\rm dim}\,(U_{\Theta(y)}(M_y)) = {\rm dim}\,M_y$ for
almost every $y$, and so ${\rm dim}\,M_y$ is almost surely invariant
under the transformation $S$, which we have assumed is ergodic.

This constancy of ${\rm dim}\,M_y$ now allows us to pick a
fiberwise orthonormal basis $\phi_1$, $\phi_2$, \ldots, $\phi_r$.
In terms of this, we know that
\[U_{\Theta(y)}({\rm span}\{\phi_1|_{Sy},\phi_2|_{Sy},\ldots,\phi_r|_{Sy}\}) = {\rm span}\{\phi_1|_y,\phi_2|_y,\ldots,\phi_r|_y\},\]
and so the unitary cocycle $U_{\Theta(\,\cdot\,)}$ specializes to
give a measurable family of $r\times r$ unitary matrices
$(U_{ij}(y))_{1 \leq i,j \leq r}$ such that
\[U_{\Theta(y)}(\phi_i|_{Sy})(z) = \sum_{j=1}^r U_{ij}(y)\phi_j|_y(z)\]
$\nu\otimes\rho$-almost surely.  However, the left-hand side of
this equation is simply $\phi_i(Sy,\Theta(y)(z))$, and so squaring
and summing in $i$ we obtain
\[\sum_{i=1}^r|\phi_i(Sy,\Theta(y)(z))|^2 = \sum_{i=1}^r\Big|\sum_{j=1}^r U_{ij}(y)\phi_j|_y(z)\Big|^2 = \sum_{j=1}^r|\phi_j(y,z)|^2,\]
by the unitarity of $(U_{ij}(y))_{1 \leq i,j \leq r}$.  It follows
that the expression $\sum_{j=1}^r|\phi_j(y,z)|^2$ is invariant
under $S_\Theta = T$, which we also assumed ergodic, and so is
almost surely constant.  Since by definition
$\int_Z|\phi_i(y,z)|^2\,\rho(dz) = 1$ for each $i=1,2,\ldots,r$,
by integrating over $z \in Z$ we see that this constant must equal
$r$.

It follows in particular that each $\phi_i$ is bounded in $L^\infty$
by this constant, and now it is immediate that membership of
$L^\infty$ persists under taking finite $L^\infty\ycn$-linear
combinations.
\end{proof}

\begin{Remark}
\em The unitary cocycle $U_{ij}$ constructed in the course of the
above proof is often referred to as the {\em relative eigenvalue}
associated to the finite-rank module $M$. Note that different
choices of basis for $M$ can give rise to different relative
eigenvalues, but they are always cohomologous.
\end{Remark}

Finite-rank modules are lent particular importance in ergodic
theory by their r\^ole in the classical dichotomy proved by
Furstenberg~\cite{Fur77} and Zimmer~\cite{Zim76.1,Zim76.2} between
relatively weakly mixing extensions and those containing a
nontrivial subextension that can be coordinatized as a compact
homogeneous skew-product extension. Although we will not need this
result in the present work, its proof for extensions over a
non-singular base is identical to that over a measure-preserving
base, now that we have Proposition~\ref{prop:finite-rank-ess-bdd}
at our disposal, and so we simply state the relevant definitions
and result here for completeness.

\begin{Def}\em
Suppose that $(X,\cb,\mu,T)$ is a non-singular system and $\ca
\subset \cb$ a factor over which it is r.f.m.p.  This factor is
{\em relatively weakly mixing} if the non-singular system
$(X\times X,\cb\otimes \cb,\mu\otimes_\ca\mu,T\times T)$ is
ergodic (this is the usual construction of the relative product
over $\ca$, which is easily seen to give a non-singular system if
$(T,X,\cb,\mu)$ is non-singular; see, for instance,~\cite{Ru-Si}).
On the other hand, it is {\em isometric} if the space
$L^\infty_\ca L^2\xbm$ is spanned by its finite rank $T$-invariant
$\ca$-submodules.
\end{Def}

\begin{Prop}
The system $(T,X,\cb,\mu)$ fails to be relatively weakly mixing over
an r.f.m.p. factor $\ca \subset \cb$ if and only if it contains a
nontrivial isometric subextension of $(T|_\ca,X/\ca,\ca,\mu|_\ca)$.
\bez
\end{Prop}

For our applications in this paper our particular interest will be
in invariant modules of rank one, whose behaviour exhibits the
following useful feature.

\begin{Prop}\label{prop:fibre-orthog}
Any two distinct rank-one $T$-invariant $\pi$-modules $M$ and $N$
are {\em fiberwise orthogonal}, in that $\E(f\ov{g}\,|\,\ca) = 0$
whenever $f \in M$ and $g \in N$.
\end{Prop}

\begin{proof}
We assume that $M$ and $N$ are  rank-one $T$-invariant
$\pi$-modules that are not fiberwise orthogonal and show that they
must actually coincide.  Indeed, in this case each admits a
one-element fiberwise orthonormal basis: that is, there are $\phi
\in M$ and $\psi \in N$ such that $M = L^\infty\ycn\cdot\phi$ and
$N = L^\infty\ycn\cdot\psi$.

Now Proposition~\ref{prop:finite-rank-ess-bdd} gives that $|\phi|$
and $|\psi|$ are both equal to $1$ almost everywhere, and that
there are maps $\sigma,\tau:Y\to\bs^1$ such that
$U_{\Theta(y)}\phi|_{Sy} = \sigma(y)\phi|_y$ and
$U_{\Theta(y)}\psi|_{Sy} = \tau(y)\psi|_y$. Defining a new
function $\xi:Y\to\C$ by
\[\xi(y) := \langle \phi|_y,\psi|_y\rangle_{L^2\zdr},\]
we compute from these relations and the unitarity of $U_{\Theta(y)}$
that
\[\xi(Sy) = \langle \phi|_{Sy},\psi|_{Sy}\rangle_{L^2\zdr} =
\langle
U_{\Theta(y)}\phi|_{Sy},U_{\Theta(y)}\psi|_{Sy}\rangle_{L^2\zdr} =
\overline{\tau(y)}\sigma(y)\xi(y).\] Since by assumption $\phi$
and $\psi$ are not fiberwise orthogonal, it follows that $\xi$ is
non-zero on a subset of $Y$ of positive measure. On the other
hand, the above shows that $|\xi(y)|$ is $S$-invariant, and so in
fact $|\xi|$ is a non-zero constant almost everywhere. Therefore
we can define $\xi'(y) := \xi(y)/|\xi(y)|\in \bs^1$, and now the
above relations together imply that the function
$\xi'(y)\overline{\phi}(y,z)\psi(y,z)$ is $T$-invariant:
\[
\xi'(Sy)\overline{\phi}(Sy,\Theta(y)(z))\psi(Sy,\Theta(y)(z))=
\big(\overline{\tau}(y)\sigma(y)\xi'(y)\big)
\big(\overline{\sigma(y)\phi(y,z)}
\big)\big(\tau(y)\psi(y,z)\big)\]\[ =
\xi'(y)\overline{\phi(y,z)}\psi(y,z).\]

Hence by the ergodicity of $T$ it must be a constant, say $c\in
\bs^1$, and so re-arranging gives $\phi(y,z) = c\xi'(y)\psi(y,z)
\in N$.  Exactly similarly we have $\psi \in M$, and so $N = M$,
as asserted.
\end{proof}

\begin{Remark}\em
In order to work with finite-rank invariant modules of higher
rank, it is necessary first to prove that any such module can be
decomposed as a direct sum of finite-rank invariant modules that
are {\em irreducible}, in that they admit no further proper
invariant submodules.  This follows relatively easily by observing
that given finite-rank invariant modules $N \subset M$, the
fiberwise orthogonal complement $M\ominus N$ defined by $(M\ominus
N)_y = M_y\ominus N_y$ is also an invariant module, and then
performing a simple induction on rank to show that repeatedly
choosing minimal sub-modules and then restricting to their
orthogonal complements leads to the desired direct sum
decomposition.

These irreducible finite-rank modules form the building blocks of
all others, but it is classically known that
Proposition~\ref{prop:fibre-orthog} does not extend to general
irreducible finite-rank invariant modules of higher rank,
even in the setting of measure-preserving actions.
In the Appendix we include for completeness
an example that witnesses this failure.

\end{Remark}

Let us now bring Proposition~\ref{prop:fibre-orthog} to bear on the
problem of growth of the set of $L^\infty$-eigenvalues.

We know that $T$ acts as an isometry on the space $L^\infty\xbm$.
We let $e(T)$ denote the group of $L^\infty$-eigenvalues of $T$:
$c\in\bs^1$ belongs to $e(T)$ if there exists $0\neq f\in
L^\infty\xbm$ such that $f\circ T=cf$. Since $T$ is also assumed
ergodic, the modulus of this $f$ is constant, whence in the
ergodic case we can additionally assume that eigenfunctions have
modulus~1. The group $e(T)$ is a Borel subgroup of $\bs^1$ and the
Borel structure is generated by a Polish topology (stronger than
the induced Euclidean topology, see \cite{Ho-Me-Pa}, \cite{Na}).
Recall also that since $L^\infty$ is not separable, in general
$e(T)$ is uncountable (see \cite{Na}, Chapter 15). Our next aim is
to prove the following.

\begin{Th}\label{rfmp3} Assume that $T$ is a non-singular ergodic
automorphism of a standard probability Borel space $\xbm$. Assume
moreover that $\ca\subset\cb$ is a factor of $T$ which is r.f.m.p.
Then the quotient $e(T)/e(T|_{\ca})$ is countable. In particular,
if $e(T|_{\ca})$ is countable, then $e(T)$ is also countable.
\end{Th}

In other words we want to prove that if $T$ has uncountably many
eigenfunctions then  all its r.f.m.p.\ factors also have
uncountably many eigenfunctions. We will see that this follows
quickly from Proposition~\ref{prop:fibre-orthog}.

\begin{proof}
Suppose that $f \in L^\infty\xbm$ is an $L^\infty$-eigenfunction
of $T$ with eigenvalue $c \in \bs^1$.  Then $M_f:=
L^\infty\ycn\cdot f$ is a rank-$1$ $\ca$-module, and since
$(h\cdot f)\circ T = c\cdot (h\circ T)\cdot f$ we see that $M$ is
$T$-invariant.  If now $g$ is another $L^\infty$-eigenfunction of
$T$ with eigenvalue $c'\neq c$, then either $M_f = M_g$, in which
case we have in particular that $f = h\cdot g$ for some $h \in
L^\infty\ycn$, hence $c f = c'(h\circ T) g = c'(h\circ
T)\overline{h}f$ and so $h$ is an $L^\infty$-eigenfunction of $S$
with eigenvalue $c\overline{c'}$; or $M_f \neq M_g$, in which case
by Proposition~\ref{prop:fibre-orthog} they are fiberwise
orthogonal in $L^\infty(Y,\cc,\nu;L^2\zdr)$.

Let $\{c_i\in e(T):\:f_i\circ T=c_i\cdot f_i,\:i\in I\}$ be a
maximal family of eigenvalues so that for $i\neq j$ the rank-$1$
$\ca$-modules $M_{f_i}$ and $M_{f_j}$ are fiberwise orthogonal. By
fiberwise orthogonality, $\E(f_i\cdot\ov{f_j}|\ca)=0$ and in
particular $f_i\perp f_j$ in $L^2\xbm$. It follows that $I$ is
countable. Moreover, by the first part of the proof
$e(T)=\bigcup_{i\in I}c_ie(T|_{\ca})$ and the result follows.
\end{proof}

Using Corollary~\ref{rfmp1} we also obtain  the following.

\begin{Cor}\label{rfmp7} Assume that $T$ is a non-singular
automorphism of a standard probability Borel space $\xbm$. Assume
that $\ca\subset \cb$ is an ergodic r.f.m.p.\ factor such that the
group $e(T|_{\ca})$ of $L^\infty$-eigenvalues is countable. Then
for almost each ergodic component $\vep$ of $T$, $e(T,\vep)$ is
also countable.
\end{Cor}

\section{Cocycles, Mackey actions and invariant measures for
Rokhlin cocycle extensions}\label{cocycles} We will now recall
some basic facts about cocycles (see e.g.\ \cite{Aa}, \cite{Sch}).
Assume that $T$ is an ergodic measure-preserving automorphism of a
standard probability Borel space $\xbm$, i.e.\ $T\in{\rm
Aut}_0\xbm$. Let $G$ be a locally compact second countable
(l.c.s.c) group. Each measurable map $\va:X\to G$ is called a {\em
cocycle}; more precisely $\va$ generates the cocycle
$\va^{(n)}(\cdot)$ by the following formula
$$
\va^{(n)}(x)=\left\{\begin{array}{lll}
\va(T^{n-1}x)\cdot\ldots\cdot\va(x)&\mbox{if}& n>0\\
1&\mbox{if}& n=0\\
(\va(T^{-1}(x)\cdot\ldots\cdot\va(T^{n}x))^{-1}&\mbox{if}&
n<0.\end{array}\right.$$ Denote by $\tf$ the corresponding skew
product:
$$
\tf:(X\times G,\cb(X\times G),\la_G)\to(X\times G,\cb(X\times
G),\la_G),\;\;\tf(x,g)=(Tx,\va(x)g),$$ where $\la_G$ is a
left-invariant Haar measure. Then
$(\tf)^n(x,g)=(T^nx,\va^{(n)}(x)g)$ and also
$\va^{(n+m)}(x)=\va^{(n)}(T^mx)\va^{(m)}(x)$ for every $n,m\in\Z$.
Let $\tau=(\tau_g)_{g\in G}$ denote the natural (right) $G$-action
on $X\times G$: $\tau_g(x,h)=(x,hg^{-1})$. Then $\tau_g$ is
non-singular with respect to $\la_G$ and it commutes with $\tf$.
Fix a probability measure $\la$ equivalent to $\la_G$ and consider
the $\sigma$-algebra $\ci_{\va}$ of $\tf$-invariant subsets. Since
$(X\times G,\cb(X\times G),\mu\ot\la)$ is a standard probability
Borel space,  the quotient space $((X\times
G)/\ci_{\va},\ci_{\va},\mu\ot\la|_{\ci_{\va}})$ is well-defined
(and is also standard). This space is called the {\em space of
ergodic components} and it will be denoted by
$(C_\varphi,\cb_\varphi,\la_\varphi)$. Since $\tau$ preserves
$\ci_\varphi$ it also acts on the space of egodic components. This
non-singular  $G$-action is called the {\em Mackey action}
associated to $\varphi$, and it is  always ergodic. It will be
denoted by $\tau^\varphi=(\tau^\varphi_g)_{g\in G}$.

A cocycle $\va:X\to G$ will be called {\em ergodic} if $\tf$ is
ergodic. This is clearly equivalent to the fact that the Mackey
action reduces to the one-point action. A cocycle $\va$ is said to
be {\em recurrent} if $\tf$ is conservative (i.e. it has no
wandering sets of positive measure). This is equivalent to the
fact that for a.e. ergodic component of $\tf$, the non-singular
action of $\tf$ on it is {\em properly ergodic} (that is, it is
not reduced to a single orbit).

Assume now that we have a measurable representation of $G$ in the
group $\Aut_0\ycn$ of measure-preserving automorphisms of $\ycn$:
$g\mapsto S_g$, and we put $\cs=(S_g)_{g\in G}$. We  define the
corresponding skew product $\tfs$ acting on $(X\times
Y,\cb\ot\cc,\mu\ot\nu)$ by $ \tfs(x,y)=(Tx,S_{\va(x)}(y))$. We are
interested in $\tfs$-invariant probability measures whose
projection on $X$ is $\mu$. The simplex of such measures will be
denoted by $ \cp(\tfs;\mu).$ On the product space $C_\varphi\times
Y$ we can consider  the diagonal  $G$-action
$\tau^\varphi\times\cs$: $(\tau^\varphi\times
\cs)_g=\tau^\varphi_g\times S_g$. Following \cite{Da-Le} (and the
earlier papers \cite{Le-Le}, \cite{Le-Pa}) we consider also the
simplex \begin{eqnarray*} &&\cp(
\tau^\varphi\times\cs,\cb_\varphi; \la_\varphi)\\
&&\quad\quad:=\{\rho\in\cp(C_\varphi\times Y)
:\:\rho|_{\cb_\varphi}=\la_\varphi\;\mbox{and}\;
\cb_\varphi\;\mbox{is an r.f.m.p.\ factor
of}\;\tau^\varphi\times\cs\}.
\end{eqnarray*}

The main result about this situation that we need is the
following.

\begin{Th}[\cite{Da-Le},\cite{Le-Pa}]\label{isomorphism} The simplices $\cp(\tfs;\mu)$
and $\cp(\tau^\varphi\times \cs,\cb_\varphi;\la_\varphi)$ are
affine isomorphic.\bez\end{Th} In what follows we will also use
some elements of the proof of this theorem.

\section{Joinings}\label{joinings}
Assume that $T\in {\rm Aut}_0\xbm$ and $S\in{\rm Aut}_0\ycn$ are
ergodic. Any $T\times S$-invariant measure $\kappa\in\cp(X\times
Y,\cb\ot\cc)$ whose projections $\kappa_X$ on $X$ and $ \kappa_Y$
on $Y$ are $\mu$ and $\nu$ respectively is called a {\em joining}
of $T$ and $S$, and we write $\kappa\in J(T,S)$. Each $\kappa\in
J(T,S)$ defines a new automorphism which is $(T\times S, X\times
Y,\cb\ot\cc,\kappa)$. Sometimes, less formally, we will also
denote the latter system as $T\vee S$. In general $\kappa\in
J(T,S)$ is not ergodic, but its ergodic decomposition consists
solely of joinings (e.g.\ \cite{Gl}), and in particular the set
$J^e(T,S)$ of ergodic joinings  is always non-empty. Note that
$\mu\ot\nu\in J^e(T,S)$ if and only if $e(T)\cap e(S)=\{1\}$.
Another easy example of an ergodic joining is available when $T$
and $S$ are isomorphic: indeed, if $W:\xbm\to\ycn$ is an
isomorphism then the measure $\mu_W$  defined by
$$\mu_{W}(B\times C)=\mu(B\cap W^{-1}(C))$$
is a member of $J^e(T,S)$. The resulting system $T\vee S$ is called
a {\em graph} joining, and is isomorphic to $T$. Following
\cite{Fu}, $T$ and $S$ are called {\em disjoint} if
$J(T,S)=\{\mu\ot\nu\}$, and in this case we write $T\perp S$.

We can also consider joinings of higher orders. If $T_i\in {\rm
Aut}_0(X_i,\cb_i,\mu_i)$ is ergodic for $i=1,\ldots,n$ then each
$T_1\times\ldots\times T_n$-invariant
$\kappa\in\cb(X_1\times\ldots\times X_n)$ whose projections
$\kappa_{X_i}$ are $\mu_i$ for $i=1,\ldots,n$ is called a joining
of $T_1,\ldots, T_n$. Such a joining is called {\em pairwise
independent} \cite{Ju-Ru} if $\kappa_{X_i\times
X_j}=\mu_i\ot\mu_j$ for $i\neq j$, $i,j=1,\ldots, n$. An ergodic
$T\in {\rm Aut}_0\xbm$ is called PID if every $\kappa\in
J(T_1,\ldots,T_n)$, with all $T_i=T$, which is pairwise
independent is equal to $\mu^{\ot n}$. The PID property was
introduced as a joining counterpart to the problem of mixing of
all orders \cite{Ju-Ru} (see  Chapter~11 in \cite{Gla03}).

Let us recall another definition in this connexion.  Recall that an
extension of systems is {\em distal} if it can be expressed as a
(possibly transfinite) tower of isometric extensions (see, for
instance, Chapter 10 of Glasner~\cite{Gla03}). Given this,
following~\cite{Ju-Le} a PID automorphism $T$ is called {\em
distally simple} (DS) if for each $\kappa\in
J^e(T,T)\setminus\{\mu\otimes\mu\}$ the system $(T\times T,X\times
X,\cb\ot\cb,\kappa)$ over the factor given by
$\cb\ot\{\emptyset,X\}\subset\cb\ot\cb$ (i.e.\ the extension $T\vee
T\to T$) is distal.

Consider now an ergodic system $T$ which has the property that
whenever we take its ergodic joining with the Cartesian product of
two weakly mixing automorphisms $S_1\times S_2$, then in the
joining $T\vee(S_1\times S_2)$ one of $S_i$s, say $S_1$, is
independent from the joining $T\vee S_2$. Such $T$ are said to
have the {\em joining primeness (JP)} property, and their basic
properties were studied in \cite{Le-Pa-Ro}. It follows from the
definition that each JP system is PID, and a little work shows
that each DS system enjoys the JP property \cite{Le-Pa-Ro}.

It can be shown that under some mild assumptions on $\va$ and
$\cs$ the extension $\tfs\to T$ is relatively weakly
mixing~\cite{Le-Le}, while any DS system must be distal over an
arbitrary non-trivial factor~\cite{Ju-Le}, and so in general the
DS property of $T$ is not retained by $\tfs$. An implicit question
in~\cite{Le-Pa-Ro} is whether under some natural assumptions on
$\va$ and $\cs$ the JP property persists under such extensions. We
will identify some such assumptions in the next section, and so
obtain natural examples of JP systems which do not enjoy the DS
property (see also \cite{Ka-Le}).

\section{Lifting multipliers of $R^\perp$} Assume that $T\in{\rm Aut}_0\xbm$ is ergodic. We
will now study ergodic properties of automorphisms of the form
$\tfs\vee S'$, i.e.\ joinings of $\tfs$ with $S'$ acting on
$(Y',\cc',\nu')$, with a view towards proving the lifting result
stated in the introduction. It is tempting to write such an
automorphism as $(T\vee S')_{\va^{S'},\cs}$ with
$\va^{S'}(x,y')=\va(x)$, but we must be aware that in this
notation it is implicit that the ``coordinate $Y$'' is independent
of $X\times Y'$ which  is of course not true in general. (To see
this it is enough to take $S'=\tfs$ and consider a graph
self-joining.) Notice however that any  joining $\kappa$ of
$\tfs$ and $S'$ is a member of $\cp((T\vee S')_{\va^{S'},\cs};
\kappa_{X\times Y'})$, where $\kappa_{X\times Y'}$ is the
projection of $\kappa$ onto $X\times Y'$.

\begin{Remark}\em Assume that $\va:X\to G$ is ergodic and $\cs$ is
uniquely ergodic. Then, as shown in \cite{Le-Le},
$$
\cp(\tfs;\mu)=\{\mu\ot\nu\},$$ i.e. $\tfs$ is relatively uniquely
ergodic. Assume now that \beq\label{ciekawostka}\kappa\in
J(\tfs,S'),\;\kappa\neq\kappa_{X\times Y'}\ot\nu.\eeq Then the
cocycle $\va^{S'}:X\times Y'\to G$ given by
$\va^{S'}(x,y')=\va(x)$ and $(T\times S',\kappa_{X\times Y'})$
cannot be ergodic. Indeed, if the cocycle $\va^{S'}$ were ergodic
we would have again that $(T\vee S')_{\va^{S'},\cs}$ is relatively
uniquely ergodic which is a contradiction
with~(\ref{ciekawostka}), since $\kappa_{X\times Y'}\ot\nu\in
\cp((T\vee S')_{\va^{S'},\cs};\kappa_{X\times Y'})$.
\end{Remark}

We will now study joinings between a probability preserving system
of the form $(\tfs,\kappa')$ with $\kappa'\in\cp(\tfs;\mu)$ (which
we will shortly denote $\tfs'$) and a system $R$ (acting on $\zdr$)
that is weakly mixing.

\begin{Lemma}\label{j1} Assume that $\kappa\in J(\tfs',R)$. If for
a.e.\ ergodic $c\in C_\varphi$ the non-singular automorphism
$(\tf,c)$ has only countably many $L^\infty$-eigenvalues then
$$
\kappa=\kappa_{X\times Y}\otimes\rho$$ provided $\kappa|_{X\times
Z}=\mu\ot\rho$ (of course $\kappa_{X\times Y}=\kappa'$).
\end{Lemma}
\begin{proof}The proof will be a small modification of the proof
of Proposition~6.1 from \cite{Le-Pa} (and also bears comparison with
the proofs of Proposition~2.1 \cite{Le-Pa} and Proposition 6.1
\cite{Da-Le}).

We use the notation $\va^R$ for the cocycle $\va$ treated as a
function defined on $X\times Z$, and so serving as a cocycle for
$(T\times R,\mu\ot\rho)$; clearly $\kappa\in\cp((T\times
R)_{\va^R,\cs};\mu\ot\rho)$. Let us write
$$\kappa=\int_{X\times Z}\delta_{(x,z)}\ot\kappa_{(x,z)}
\,d\mu(x)\,d\rho(z)$$ and put
$\ov{\kappa}_{(x,g,z)}=S_g\kappa_{(x,z)}$; these measures define
the measure $\ov{\kappa}$ by
$$\ov{\kappa}=\int_{X\times Z\times
G}\ov{\kappa}_{(x,g,z)}\,d\kappa|_{X\times Z}(x,z)\,d\lambda(g).$$
Finally, the isomorphism in Theorem~\ref{isomorphism} sends $\kappa$
to $\widetilde{\kappa}$ which is the projection of $\ov{\kappa}$ on
$C_{\va^R}\times Y$. The map $$(x,g,z)\mapsto
\ov{\kappa}_{(x,g,z)}$$ is $(T\times R)_{\va^R}$-invariant (see the
formula~(10) in \cite{Le-Pa}), that is it is
$\ci_{\va^R}$-measurable. By assumption, on a.e.\ ergodic component
$c\in C_\varphi$ the (non-singular) automorphism $\tf$ is ergodic
and has only countably many $L^\infty$-eigenvalues. Since $R$ is
weakly  mixing, the non-singular Cartesian product automorphism
$(\tf|_c)\times R$ is still ergodic (this follows, for instance,
from the spectral condition of Theorem~2.7.1 in~\cite{Aa}, since the
spectral type of our weakly mixing transformation $R$ must
annihilate the countable set $e(T_\phi|_c)$). It follows that
$\ci_{\va^R}=\ci_\va\ot\{\emptyset,Z\}$. Hence the map
$(x,g,z)\mapsto \ov{\kappa}_{(x,g,z)}$ is in fact a function of
$(x,g)$ alone, and upon integrating
$$
\kappa_{(x,z)}=\int_GS_g^{-1}\ov{\kappa}_{(x,g,z)}\,d\la(g)$$ is a
function of $x$ alone, whence the result.
\end{proof}

We have now reached the proof of the main result:

\begin{Th}\label{main}
Assume that $T$ is an ergodic automorphism of a standard probability
Borel space $\xbm$. Let $\va:X\to G$ be an ergodic cocycle such that
$e(\tf)$ is countable. Let $\cs=(S_g)_{g\in G}$ be an ergodic
representation of $G$ in $\Aut_0\ycn$ (so $\tfs$ is ergodic). Assume
that $R$ is a weakly mixing automorphism of $\zdr$. If $T\in
\cm(\{R\}^\perp)$ then
$$
\tfs\in\cm(\{R\}^\perp).$$
\end{Th}
\begin{proof}
Take $\kappa\in J^e(\tfs,S',R)$ where $S'$ is an ergodic
automorphism acting on $(Y',\cc',\nu')$ and $S'\perp R$. Consider
first the automorphism $(T\times S',\kappa|_{X\times
Y'})_{\va^{S'}}$. Notice that $\tf$ is an r.f.m.p.\ factor (via
the map $(x,y',g)\mapsto (x,g)$) and that the cocycle $\va^{S'}$
is recurrent. Since $e(\tf)$ is countable, by
Corollary~\ref{rfmp7} for a.e.\ ergodic component $c\in
C_{\va^{S'}}$ the non-singular ergodic automorphism $((T\vee S')
_{\va^{S'}},c)$ has only countably many $L^\infty$-eigenvalues.

Consider $\kappa$ as an element of $J^e((T\vee
S')'_{\va^{S'},\cs},R)$, where, by our standing assumption,
$$
\kappa|_{X\times Y'\times Z}=\kappa|_{X\times Y'}\ot\rho.$$ It now
follows from Lemma~\ref{j1} that $\kappa=\kappa|_{X\times Y'\times
Y}\ot \rho$ and the result follows.
\end{proof}

This  strengthens  a result from \cite{Da-Le} where the lifting
multiplier theorem was proved for $T$ an ergodic rotation.

Notice also that from the proof of Theorem~\ref{main} we easily
deduce the following (see the end of Section~\ref{joinings}).

\begin{Cor}\label{jp100} Under
the assumptions on $\va$ and $\cs$ in Theorem~\ref{main}, if $T$
enjoys the JP (respectively, PID) property then  so does
$\tfs$.\bez
\end{Cor}

\appendix

\section{Appendix: Non-orthogonal irreducible rank-$2$ modules}
We will show that Proposition~\ref{prop:fibre-orthog} does not
extend to irreducible rank-2 modules in the setting of
measure-preserving systems.

 We first construct the following extension of systems.
Let $(S,Y,\cc,\nu)$ be any aperiodic ergodic $\nu$-preserving
system. Then let $\theta:Y\to {\rm U}(2)$  (the group of $2\times
2$ unitary matrices considered with Haar measure $\la_{{\rm
U}(2)}$) be any ergodic cocycle. Let $\bs^3$ denote the unit
sphere in $\C^2$. Take $z_0=(1,0)\in \bs^3$ and fix $w\in {\rm
U}(2)$, so that \beq\label{e}w_{21}\neq0\eeq and also
$\alpha:=w_{11}=\langle z_0,wz_0\rangle_{\C^2}\neq0$. Then the set
$$
\{(uz_0,uwz_0):u\in {\rm U}(2)\}\!=\! \{(z_1,z_2)\in \bs^3\times
\bs^3: \langle z_1,z_2\rangle_{\C^2} = \alpha\} \left(\subset
{\bs^3}\times {\bs^3}\right)$$ is a hypersurface with the measure
$\rho$ which is the image of $\la_{{\rm U}(2)}$ via the map
$u\mapsto (uz_0,uwz_0)$.

Form the extended space
$$\xbm := (Y,\cc,\nu)\otimes ({\rm U}(2),\cb({\rm U}(2)),\la_{{\rm U}(2)})$$ and define the transformation
$T$ on $X$ by
\[T(y,u) = (Sy,\theta(y)u).\] It follows that
$T$ preserves $\mu = \nu\otimes \la_{{\rm U}(2)}$, and it is
ergodic by our assumption on the cocycle $\theta$.

Now let $M$ and $N$ be the finite-rank $\pi$-submodules of
$L^2\xbm$ defined by the respective bases $\{\phi_1,\phi_2\}$ and
$\{\psi_1,\psi_2\}$, where
\[\phi_i(y,u) = (uz_0)_i\quad\quad\mbox{and}
\quad\quad \psi_i(y,u) = (uwz_0)_i,\] writing $x_i$ for the
$i^{{\rm th}}$ component of $x \in \bs^3$ for $i=1,2$. It is now
easily checked that the each of $M$ and $N$ is $T$-invariant, and
their irreducibility follows from the fact that $\theta$ is
ergodic.

Indeed, since $M$ has rank $2$, if it were not irreducible then it
would contain a rank-$1$ submodule $M'$. Suppose that $F$ is the
base of $M'$. Then
$$F(y,u)=h_1(y)(uz_0)_1+h_2(y)(uz_0)_2$$
and, by $T$-invariance, $F\circ T=f(y) F $ (with $|f|=1$). However
both sides of the latter equality depend only on $(y,uz_0)$, so if
we consider $F=F(y,uz_0)$ then
$$F(Sy,(\theta(y)u)z_0)=f(y)F(y,uz_0).$$ Putting
$H(y)=(\overline{h_1(y)},\overline{h_2(y)})$ we can rewrite this
as
$$ F(y,u)=\langle uz_0,H(y)\rangle.$$ Hence
$$
\langle \theta(y)uz_0,H(Sy)\rangle=f(y)\langle uz_0,H(y)\rangle.$$
Since $uz_0$ runs over all vectors of length 1 as $u$ runs over
${\rm U}(2)$, we must have
$\theta(y)^{-1}H(Sy)=\overline{f(y)}H(y)$. This can be re-written
for the function $\tilde{H}(y,u)=u^{-1}H(y)$ as $\tilde{H}\circ
S_\theta=\overline{f(y)}\tilde{H}$. But
$$\tilde{H}(y,u)=(\tilde{H}_1(y,u),\tilde{H}_2(y,u)),$$
so $ \tilde{H}_i\circ S_\theta=\overline{f(y)}\tilde{H}_i$ for
$i=1,2$. Since $f$ is of modulus one and $S_\theta$ is ergodic,
there is a constant $c\in\C$ such that $\tilde{H}_1=c\tilde{H}_2$
which is impossible because of the definition of $\tilde{H}$.

We will now show that as subspaces of $L^2\xbm$ we have $M\cap N =
\{0\}$. If not, we can find $g_1,g_2,h_1,h_2$ elements of
$L^\infty\ycn$ such that \beq\label{n1} g_1(y)(uz_0)_1 +
g_2(y)(uz_0)_2=h_1(y)(uwz_0)_1+h_2(y)(uwz_0)_2\eeq a.e.\ for the
product measure $\nu\ot\la_{{\rm U}(2)}$, so for $\nu$-a.e.\ $y\in
Y$ we have the above equality for $\la_{{\rm U}(2)}$-a.e.\ $u \in
{\rm U}(2)$. To show that~(\ref{n1}) does not hold we take
$(a,b)\in\C^2\times\C^2$ and let
$$
W:=\{u\in {\rm U}(2):\: \langle
(uz_0,uwz_0),(a,b)\rangle_{\C^4}=0\}$$$$=\{u\in {\rm
U}(2):\:\langle((u_{11},u_{21}),(u_{11}w_{11}+u_{12}w_{21},u_{21}w_{11}+u_{22}w_{21})),
(a,b)\rangle_{\C^4}=0\}.$$ We need to show that $\la_{{\rm
U}(2)}(W) = 0$ (unless $a = b = 0$). Suppose instead that
$\la_{U(2)}(W)>0$. By letting $\bs^1$ act on ${\rm U}(2)$ as left
translations by the matrices $\left(\begin{array}{ll}1&
0\\0&\la\end{array}\right)$ we obtain that  for a.e. $u\in W$
there are infinitely many $\la$, $|\la|=1$, such that
$$\left(\begin{array}{ll}1&
0\\0&\la\end{array}\right)u\in W.$$ Therefore
$$
\langle((u_{11},\la u
_{21}),(u_{11}w_{11}+u_{12}w_{21},\la(u_{21}w_{11}+u_{22}w_{21}))),
(a,b)\rangle_{\C^4}=0.$$ Because this is linear in $\la$,
$$
\langle (u_{11},u_{12}),
(a_1+\ov{w}_{11}b_1,\ov{w}_{21}b_1)\rangle=0$$ and
$$\langle (u_{21},u_{22}),
(a_2+\ov{w}_{11}b_2,\ov{w}_{21}b_2)\rangle=0.$$ Since the latter two
equalities are satisfied on a set of $u\in U(2)$ of positive
measure, $a_1+\ov{w}_{11}b_1=0=\ov{w}_{21}b_1$ and also
$a_2+\ov{w}_{11}b_2=0=\ov{w}_{21}b_2$. In view of (\ref{e}) this
implies $a_1=a_2=b_1=b_2=0$.

Finally, the function
$$u\mapsto\overline{(uz_0)_1}(uwz_0)_1 +
\overline{(uz_0)_2}(uwz_0)_2 = \langle u z_0, uwz_0\rangle_{\C^2}
= \alpha$$ is constant on ${\rm U}(2)$ and therefore
\[\langle\phi_1|_y,\psi_1|_y\rangle_{L^2\zdr} +
\langle\phi_2|_y,\psi_2|_y\rangle_{L^2\zdr}\]\[ = \int_{{\rm
U}(2)} \overline{(uz_0)_1}(uwz_0)_1 +
\overline{(uz_0)_2}(uwz_0)_2\,\la_{{\rm U}(2)}(d u) =
\alpha\neq0,\] so the modules $M$ and $N$ are not fiberwise
orthogonal.  This completes our example.

Notice that in the above example, although we have constructed two
distinct irreducible modules $M$ and $N$ that are not fiberwise
orthogonal, their associated relative eigenvalues (the unitary
cocycles $U_{ij}$) are cohomologous (indeed, they are both equal
to $\theta$). In fact this is a general phenomenon: if two
irreducible finite-rank modules have non-cohomologous relative
eigenvalues, then it does hold that they must be fiberwise
orthogonal.  This is shown in the case of a measure-preserving
base system by Thouvenot in~\cite{Th1}; we quickly recall the
argument here for completeness.

We show that if $M$ and $N$ are irreducible but not fiberwise
orthogonal and we adopt respective fiberwise orthonormal bases
$\phi_1$, $\phi_2$, \ldots, $\phi_r$ and $\psi_1$, $\psi_2$,
\ldots, $\psi_s$, then the relative eigenvalues obtained from
expressing the restrictions of the overall cocycle $U_\Theta$ in
terms of these bases are cohomologous.  To see this, first let
$P_y:M_y \to N_y$ be the restriction to $M_y$ of the orthogonal
projection $L^2(\mu_y)\to N_y$.  This is easily seen to depend
measurably on $y$.  Our assumption that $M$ and $N$ are not
fiberwise orthogonal now implies that $P_y$ is not identically
zero.  However, we also clearly have $P_{Sy}\circ
U_{\Theta(y)}|_{M_y} = U_{\Theta(y)}|_{N_y}\circ P_y$, as a
consequence of the invariance of $M$ and $N$, and from this it
follows that $y \mapsto {\rm im}P_y$ and $y \mapsto {\rm ker}P_y$
define invariant submodules of $N$ and $M$ respectively, which
must therefore be trivial by irreducibility.  Since $P_y$ is
non-zero, this implies that it is both almost surely surjective
and almost surely injective, and hence that it is almost surely an
isomorphism.

We can extend this conclusion about $P_y$ as follows: for any
Borel $I \subset \R$, the sum of the eigenspaces of $P_y^\ast P_y$
(respectively, of $P_yP_y^\ast$) corresponding to eigenvalues in
$I$ also defines an invariant submodule of $M$ (respectively, of
$N$), and so by irreducibility must be either full or trivial.
This implies that both $P_y^\ast P_y$ and $P_yP_y^\ast$ are
actually almost surely scalar multiples of the identity operator
(on $M_y$ and $N_y$, respectively), and hence that $P_y$ is almost
surely a scalar multiple of an isometry, say $P_y = \alpha_y
\Phi_y$ with $\alpha_y>0$.  Now using this expression for $P_y$ in
the equation above gives $\alpha_{Sy}\Phi_{Sy}\circ
U_{\Theta(y)}|_{M_y} = \alpha_y U_{\Theta(y)}|_{N_y}\circ \Phi_y$,
and since $\Phi_{Sy}\circ U_{\Theta(y)}|_{M_y}$ and
$U_{\Theta(y)}|_{N_y}\circ \Phi_y$ are both isometries this
requires that $\alpha_{Sy} = \alpha_y$ and $\Phi_{Sy}\circ
U_{\Theta(y)}|_{M_y} = U_{\Theta(y)}|_{N_y}\circ \Phi_y$.
Finally, expressing this last equality in terms of the bases
$\phi_i$ and $\psi_j$ gives the desired cohomology, since
$U_{\Theta(y)}|_{M_y}$ and $U_{\Theta(y)}|_{N_y}$ are expressed as
the two relative eigenvalues and $\Phi_y$ becomes a unitary
matrix, since it is an isometry expressed between two orthonormal
bases.

\subsubsection*{Acknowledgements}

This paper was written while both authors were visiting MSRI
(Berkeley) in the Fall 2008.  Our thanks go to them for their
hospitality, and to Alexandre Danilenko for suggesting a valuable
improvement to the proof of Theorem~\ref{rfmp3}. We would also like
to thank Emmanuel Lesigne for numerous helpful discussions and
suggestions both on the content and on improving the presentation of
the paper.

Tim Austin, Department of Mathematics,
University of California, Los Angeles,
California 90095-1555, U.S.A.; timaustin@math.ucla.edu\\
Mariusz Lema\'nczyk, Faculty of Mathematics and Computer Science,
Nicolaus Copernicus University, 87-100 Toru\'n, Poland;
mlem@mat.uni.torun.pl

\end{document}